\documentclass[11pt,reqno]{amsart}

\usepackage[T1]{fontenc}
\usepackage{fourier,microtype,booktabs,hyperref}
\usepackage{amsthm,amsmath,amssymb}
\usepackage[font=footnotesize]{caption}

\newcommand{\lab}[1]{\label{#1}} 
\newenvironment{Quote}{\list{}{\leftmargin=1.1cm\rightmargin=1.1cm}\item[]}{\endlist}

\newcommand{\rational}{\ensuremath {\mathbb Q} }
\newcommand{\real}{\ensuremath {\mathbb R} }

\newcommand{\eps}{\epsilon} 
\newcommand{\cal}{\mathcal} 

\newcommand{\be}{\begin{equation}}
\newcommand{\ee}{\end{equation}}
\newcommand{\bea}{\begin{eqnarray}}
\newcommand{\eea}{\end{eqnarray}}
\newcommand{\bean}{\begin{eqnarray*}}
\newcommand{\eean}{\end{eqnarray*}}
\newcommand\eqn[1]{(\ref{#1})}
\newcommand{\bel}[1]{\be\lab{#1}}

\newtheorem{thm}{Theorem}
\newtheorem{lem}[thm]{Lemma}
\newtheorem{cor}[thm]{Corollary}
\newtheorem{prop}[thm]{Proposition}

\author[B. Kolesnik]{Brett Kolesnik}
\author[N. Wormald]{Nick Wormald}

\address{Brett Kolesnik
\hfill\break Department of Mathematics, University of British Columbia
}
\email{\href{mailto:bkolesnik@math.ubc.ca}{bkolesnik@math.ubc.ca}}
  
\address{Nick Wormald
\hfill\break Department of Mathematical Sciences, Monash University
}
\email{\href{mailto:nick.wormald@monash.edu}{nick.wormald@monash.edu}}

\begin{document}

\title[Isoperimetric numbers of random regular graphs]
{Lower bounds for the isoperimetric numbers of random regular graphs}

\begin{abstract}
The vertex isoperimetric number of a graph $G=(V,E)$
is the minimum of the ratio $|\partial_{V}U|/|U|$
where $U$ ranges over all nonempty subsets of $V$ with $|U|/|V|\le u$
and $\partial_{V}U$ is the set of all vertices adjacent to $U$ but
not in $U$. The analogously defined edge isoperimetric number\textemdash{}with 
$\partial_{V}U$ replaced by $\partial_{E}U$, the set of all
edges with exactly one endpoint in $U$\textemdash{}has been studied extensively.
Here we study random regular graphs. For the case $u=1/2$, we give 
asymptotically almost sure lower bounds for the vertex isoperimetric 
number for all $d\ge3$. Moreover, we obtain a lower bound on the 
asymptotics as $d\to\infty$. We also provide asymptotically almost sure 
lower bounds on $|\partial_{E}U|/|U|$ in terms of an upper bound on 
the size of $U$ and analyse the bounds as $d\to\infty$.
\end{abstract}

\maketitle

\section{Introduction}\lab{s:intro}

\thispagestyle{empty}
In this paper we consider versions of the isoperimetric number of random regular graphs. 
These are explicit indicators of the notion generally called expansion  
(see below for the relevant definitions). Random regular graphs give nondeterministic 
examples of expander graphs, and as mentioned in~\cite[Section 4.6]{HLW}, there is 
great interest in the edge and vertex expansion of sets of varying sizes. Here we 
obtain explicit bounds on the expansion of sets with given size  in random regular graphs. 
We concentrate on the vertex version, which is more difficult and less well studied 
than the edge version.

Let  $G$ be a graph on $n$ vertices. For a subset $U$ of its vertex set $V=V(G)$, let 
$\partial_{V }(U)$ denote the set of all vertices adjacent to a vertex in $U$ but not in $U$. 
Similarly, let $\partial_{E }(U)$ denote the set of
all edges with exactly one end  in   $U$. 
Note that $|\partial_{V }(U)|\leq|\partial_{E }(U)|$.
For any $0<u\le1/2$ the 
\emph{$u$-edge isoperimetric
number} is defined as
\[
i_{E,u}(G)=\min_{|U|\leq un}|\partial_{E}(U)|/|U|,
\]
and likewise for any $0<u\le1$  the \emph{$u$-vertex isoperimetric number} 
of $G$ as 
\[
i_{V,u}(G)=\min_{|U|\leq un}|\partial_{V }(U)|/|U|.
\]

The $1/2$-edge ($1/2$-vertex) isoperimetric number is often referred
to as the {\em edge (vertex) isoperimetric number}, and in this case, we simplify 
notation as $i_E(G)$ and $i_V(G)$. For the edge version, this makes immediate 
sense since a lower bound on the number of edges joining $S\subseteq V $ to 
$V\setminus S$ is obtained as $i_{E}(G)\rho  $ where $\rho=\min\{|S|,|V\setminus S|\} $. 
For the vertex isoperimetric number, the situation is not quite so symmetrical since 
$|\partial_{V }(U)|\ne |\partial_{V }(V\setminus U)|$ in general. However, $u=1/2$  has 
some uses. Note that the iterated neighbourhoods of any vertex in $G$ expand by (at least) 
a  factor $\alpha = 1+i_{V}(G)$ until they reach size $n/2$ (or more). Hence an easy upper 
bound on the diameter of $G$ is $2\log_\alpha(n/2)$.
To give another example, Sauerwald and Stauffer~\cite{SS11} recently showed that if a 
certain rumour spreading process takes place on a regular graph, where informed 
vertices randomly select a neighbour to inform (i.e., the push model), then all vertices are 
informed asymptotically almost surely (a.a.s.) after $O((1/i_{V})\cdot\log^5{n})$ steps 
of the process. Hence a lower bound on $i_{V}$ gives a upper bound (holding a.a.s.) 
for the time at which all vertices are aware of the rumour. 

A graph is $d$-regular if all its vertices are of degree $d$. Let
${\cal G}_{n,d}$ denote the uniform probability space on the set
of all $d$-regular graphs on $n$ vertices that are simple 
(i.e., have no loops or multiple edges). 
A property holds a.a.s.\ in a sequence
of probability spaces on $\{\Omega_{n}\} $ if
the probability that an element of $\Omega_{n}$
has the given property converges to 1 as $n\to\infty$.
Define
\[
i_{V,u}(d)=\sup\left\{ \ell:i_{V,u}({\cal G}_{n,d})\geq \ell\mbox{  a.a.s.}\right\}
\]
and define $i_{E,u}(d)$ similarly. In the case $u=1/2$, we simply write 
$i_V(d)$ and $i_E(d)$.

We can now describe our results more explicitly. Our main purpose is to provide 
asymptotically almost sure lower bounds for the vertex expansion of random 
regular graphs.

In Section~\ref{s:history}, we highlight some results in the literature that relate to 
vertex and edge expansion of regular graphs. 

The pairing model, as used by Bollob{\'a}s~\cite{Bol88} to investigate $i_{E}(d)$, 
i.e., for the case of edge expansion, is discussed in Section~\ref{s:models}. 
As we shall see, this model is also helpful for studying $i_{V,u}(d)$. 

In Section~\ref{s:vertparam}, we introduce the method we use. In short, we obtain 
lower bounds on vertex expansion using the first moment method: for a sequence 
of non-negative, integer-valued random variables $\{X_n\}$, provided that 
${\bf E}(X_n)\to0$ as $n\to\infty$, it follows that $X_n=0$ a.a.s.\ For an introduction 
to probabilistic techniques in discrete mathematics, refer to~\cite{AS08}, for instance. 
In~\cite{Bol88}, the first moment method easily provides lower bounds on $i_{E}(d)$. 
However, in bounding $i_{V,u}(d)$, the method yields bounds which are initially quite 
opaque. The main complication is that both $\partial_V$ and $\partial_E$ need to 
be considered. For a sequence $u(n)\to u_0\in[0,1/2]$ and numbers $s$ and $y$, 
we define a random variable $X$ that counts the number of subsets $U$ with  
$|U|=un$, $|\partial_V|=sn$ and $|\partial_E|=yn$ in a graph from ${\cal G}_{n,d}$.  
For such a sequence, we use the first moment method to determine the range of 
$s$ for which $X=0$ a.a.s.\ This leads us to define a 
\emph{$u$-vertex expansion number} $I_{V,u}(d)$,
which can be thought of as an asymptotic profile for the vertex expansion of subsets 
$U\subset V({\cal G}_{n,d})$ with $|U|\sim un$.
(See Section~\ref{s:vertparam} for the precise definition.)

We obtain lower bounds $A_d(u)$ on $I_{V,u}(d)$ in Section~\ref{s:vertparam}, 
and in Table~\ref{tablev} we provide approximate values for $A_d(u)$ for several 
values of $d$ and $u$. Since one expects the isoperimetric number of a graph to 
be obtained by larger sets, it is reasonable to conjecture that the $u$-isoperimetric 
and expansion numbers coincide, i.e., $i_{V,u}(d)=I_{V,u}(d)$, and hence, for all 
$d\ge3$ and $u\in(0,1/2]$, $i_{V,u}(d)\ge A_d(u)$. Unfortunately, the formulae do 
not seem to have a convenient explicit form,
and so, this is not straightforward to show for the cases $u<1/2$.  
We plan to address $i_{V,u}(d)$ for $u<1/2$ at a later time. 

In the present work, we deal with the case $u=1/2$. In Section~\ref{s:uhalf} we 
obtain explicit lower bounds on the vertex isoperimetric number $i_{V}(d)$ 
for all $d\ge3$. 
 
\begin{thm} \label{t:mainvertexhalf}
For $d\ge 3$, 
\[
i_{V}(d)\ge A_d(1/2)=2s_d
\]
where $s_d$ is the smallest positive solution to
\[
(2^d-1)^{s}=2^{d/2+s-1}(1-2s)^{1/2-s}{s}^{s}.
\]
\end{thm}

Table~\ref{tablevhalf} provides approximate values for $A_d(1/2)$ for several 
values of $d$.

In Section~\ref{s:asymp} we apply Theorem~\ref{t:mainvertexhalf} to obtain a 
lower bound on the asymptotics of $i_V(d)$ as $d\to\infty$.

\begin{cor} \label{t:1/2asymp}
As $d\to\infty$,
\[
i_V(d)\ge 1-2/d +O((\log{d})/d^2).
\]
\end{cor}

Corollary~\ref{t:1/2asymp} improves upon the information that is otherwise 
available from spectral results. See Section~\ref{s:history} for a discussion on this.

We switch our attention to the edge isoperimetric number in Section~\ref{s:edge}. 
Bollob\'as~\cite{Bol88} computed lower bounds with the first moment method for 
$i_{E}(d)$, i.e., for the case of edge expansion at $u=1/2$, for all $d\ge3$. Therein 
it is shown that for sufficiently large $d$,
\bel{eAsymt} 
i_{E}(d)\ge d/2-\sqrt{d\log2}\ee
and so
\[
\lim_{d\to\infty}i_{E}(d)/d=1/2.
\]
Further, it is noted that for $n\ge d+1$, if $G$ is a $d$-regular graph in $n$ vertices 
and $U$ is selected uniformly from $\{1,2,\ldots,n\}$ such that $|U|=\lfloor n/2\rfloor$, 
then
\[
{\bf E}(\partial_V U)=d\lfloor n/2\rfloor\lceil n/2\rceil/(n-1).
\]
Hence, the lower bound \eqn{eAsymt} is, 
as stated in~\cite{Bol88}, `essentially best possible' for large $d$.
Then, more generally, it is claimed in~\cite{Bol88} that
\[
\lim_{d\to\infty}i_{E,u}(d)/d=1-u
\]
for all $0<u<1/2$; however, no explicit lower bounds are given. 
Our contribution is to find lower bounds on $i_{E,u}(d)$ for all $d\ge3$ and $0<u\le1/2$. 
The argument is similar to that of~\cite{Bol88} in that we compute the expected number of 
sets in a random regular graph with an edge boundary of a given size and then apply the 
first moment method.  However, our use of 
the \emph{$u$-edge expansion number} $I_{E,u}(d)$ (analogous to $I_{V,u}(d)$) 
yields a simple method for computing `best possible' lower bounds on $i_{E,u}(d)$ for 
all $d\ge3$ and $0<u\le1/2$. The following result is proved in Section~\ref{s:edge}.

\begin{thm} \label{t:eThm}
For $d\ge3$ and $u\in(0,1/2]$,
\[
i_{E,u}(d)\ge \widehat A_d(u)= y_{d,u}/u
\]
where $y_{d,u}$ is the smallest positive solution to
\[
d^{d/2}u^{(d-1)u}(1-u)^{(d-1)(1-u)}=(du-y)^{(du-y)/2}(d-du-y)^{(d-du-y)/2}y^y.
\]
\end{thm}

Table~\ref{etable} provides approximate values for $ \widehat A_d(u)$ for several 
values of $d$ and $u$.

Applying Theorem~\ref{t:eThm}, in Section~\ref{s:edge} we obtain lower bounds 
on the asymptotics of $i_{V,u}(d)$, as $d\to\infty$, for all $0<u\le1/2$.

\begin{cor}  \lab{c:eCor} 
Fix $u\in(0,1/2]$. As $d\to\infty$, 
\[
\widehat A_d(u) = d(1-u)-2(1-u)\sqrt{d\log(u^{-u}(1-u)^{u-1})} +o(\sqrt d).
\]
Hence
\[
i_{E,u}(d)\ge d(1-u+o(1))\quad (\mbox{as } d\to\infty).
\]
\end{cor}

\section{Background}  \lab{s:history}

A $u$-edge ($u$-vertex) $\alpha$-expander, $\alpha>0$,
is a graph that has a $u$-edge ($u$-vertex) isoperimetric
number at least as large as $\alpha$. The $u$-isoperimetric
numbers of a graph may be viewed as indicators for its level of connectivity.
Expander graphs provide a wealth of theoretical interest and have
many applications. For a thorough exposition of the theory and applications
of expander graphs, see~\cite{HLW} and the references therein.

Regular graphs are known to be good expanders with high probability;
however, determining the isoperimetric number with precision is a
difficult task. For instance, as shown by Golovach~\cite{Gol94}, even
the problem of determining whether $i_{E}(G)\le q$,
provided $q\in\rational$ and $G$ has degree sequence bounded by 3,
is ${\bf NP}$-complete. Thus, bounds for the isoperimetric numbers
of regular graphs are of interest.

Buser~\cite{Bus84} showed that for all $n\ge4$ there exists
a cubic ($3$-regular) graph on $n$ vertices whose edge isoperimetric
number is at least $1/128$.  To quote Bollob\'as~\cite{Bol88}:
\begin{Quote}
Buser's proof $\ldots$ was very unorthodox in combinatorics and very exciting: 
it used the spectral geometry of the Laplace operator on Riemann surfaces, 
Kloosterman sums and the Jacquet\textendash{}Langlands theory. 
As Buser wrote: the proof `is rather complicated and it would be more 
satisfactory to have  an elementary proof.'
\end{Quote}
Using the standard first moment method, Bollob\'as~\cite{Bol88} provided a 
simple proof that much more is true. In~\cite{Bol88} it is shown that in fact
\[
i_{E}(d)\ge d/2-\sqrt{d\log 2} \quad(\mbox{as } d\to\infty).
\]
Several bounds for small $d$ are provided
in~\cite{Bol88}, such as $i_{E}(3)>2/11$,
$i_{E}(4)>11/25$, and $i_{E}(5)>18/25$.
Other bounds have been found by analysing the spectral gap of the
adjacency matrix of regular graphs. Let $\lambda(G)$ denote
the second largest (i.e., largest other than $d$) eigenvalue
in the adjacency matrix of $G$. Alon and Millman~\cite{AM85}
proved that if $G$ is $d$-regular, then \[
i_{E}(G)\ge(d-\lambda)/2.\]
Further, Alon~\cite{Alo97} showed if $n>40d^{9}$ and $G$ is
$d$-regular, then \[
i_{E}(G)\leq d/2-3\sqrt{d}/16\sqrt{2}.\]
The best lower bound to date for cubic graphs, $i_{E}(3)\ge1/4.95$,
was found by    Kostochka
and Melnikov~\cite{KM93} using an edge weighting argument. 
For the cases $d>3$, Lampis~\cite{L14} has obtained improvements 
on the lower bounds
for $i_E(d)$ given in~\cite{Bol88} by investigating subsets 
which are, in a certain sense, locally optimal.
Upper bounds for $i_{E}(d)$
are available via results for the bisection width of regular
graphs. Note that the (edge) bisection width of a graph $G=(V,E)$
is defined as \[
b(G)=\min_{(n-1)/2\le|U|\le n/2}|\partial_{E}(U)|/|U|,\]
and so clearly for any graph $G$, 
\[
i_{E}(G)\le2 b(G).
\]
The best upper bound for cubic graphs is that of Monien
and Preis~\cite{MP01}. Therein it is proved $ b(G)\le 1/6+\epsilon$ 
for all $\epsilon>0$ for {\em all} sufficiently large connected cubic graphs $G$. 
(The required lower bound on the size of $G$ depends on $\epsilon$.) 
Hence $i_{E}(3)\le1/3$. This bound was found by way of an algorithmic 
procedure which attempts to find a small cut.
The best upper bounds to date for $d>3$ have been found by 
D\'{i}az, Do, Serna, and Wormald~\cite{DDSW} 
and D\'{i}az, Serna, and Wormald~\cite{DSW} 
via randomized greedy algorithms analysed using the differential equation method. 
From the results of~\cite{DDSW} it follows that $i_{E}(4)\le2/3$,
and from the numerical values in~\cite{DSW} we obtain further values 
such as $i_{E}(5)\le1.0056$,
$i_{E}(6)\le1.3348$ and $i_{E}(7)\le1.7004$.

In comparison, less information about the vertex isoperimetric numbers
of regular graphs is available. Of course if $G$ is $d$-regular,
then \[
i_{V,u}(G)\le i_{E,u}(G)\le d\cdot i_{V,u}(G)\]
for all $0<u\le1/2$; but these bounds are far from sharp for most values of $u$.
Some interesting results are as follows. Tanner~\cite{Tan84} proved that if 
$G\in{\cal G}_{n,d}$ and $\lambda(G)\le\alpha d$,
then  
\[
i_{V,u}(G)\ge1/\big(u(1-\alpha^{2})+\alpha^{2}\big)-1.
\]
Friedman~\cite{Fri03} showed by a detailed examination of numbers of closed walks that 
\[
\lambda(G)\le2\sqrt{d-1}+\epsilon
\]
a.a.s.$\!$ in ${\cal G}_{n,d}$ for any $\epsilon>0$. Thus it follows
for any $d\ge3$ and $0<u\le1/2$, 
\[
i_{V,u}(d)\ge1/\big(u(1-4(d-1)/d^{2})+4(d-1)/d^{2}\big)-1.
\]
So, in particular, \bel{asymp1}
i_{V}(d)\ge1-8/d+O(1/d^{2}).\ee
Finally, one other result of interest
concerns the expansion of small sets; see~\cite[Theorem 4.16.1]{HLW}. 
For any $d\ge3$ and $\delta>0$, there exists an
$\epsilon_{\delta}>0$ for which
\bel{smallsets}i_{E,\epsilon_{\delta}}(G)\ge d-2-\delta\ee
 a.a.s.$\!$ in ${\cal G}_{n,d}$.
In fact, the same is true of vertex expansion. In what follows, 
we refer to the above as the \emph{small sets property}. In both cases, 
the expansion parameter $d-2$ is best possible. In particular, for any $k$,
\bel{isoub}
\min_{|U|=k}|\partial_V U|/|U|\le\min_{|U|=k}|\partial_E U|/|U|\le d-2+2/k.
\ee
For details see~\cite[Subsection 5.1.1]{HLW}.

\section{Model for analysis}\lab{s:models}

To analyse ${\cal G}_{n,d}$ we make use of the pairing model, 
${\cal P}_{n,d}$, described as follows. Suppose there are $n$ cells, each
containing $d$ points, where $dn$ is even. 
Let ${\cal P}_{n,d}$ denote the uniform probability space on the set 
of all perfect matchings of the $dn$ points. 
By collapsing each cell of a given $H\in{\cal P}_{n,d}$ into a single
vertex, a \emph{d}-regular multigraph $\pi_H$ on $n$ vertices is obtained.
The pairing model is due to Bollob\'as, who was the first to suggest directly 
deducing properties of random graphs from the model, though similar 
models appear in earlier works (see~\cite{Wor99} for details).

It is known that ${\bf P}(\pi_{H}\mbox{ is simple})$ is
bounded away from $0$ as $n$ tends to infinity, and that all $d$-regular 
simple graphs are selected with equal probability through the process
of choosing an $H\in{\cal P}_{n,d}$ uniformly at random and then
constructing $\pi_{H}$. Thus, to prove that a property occurs a.a.s.$\!$ in 
${\cal G}_{n,d}$, it is enough to prove that the pairings in ${\cal P}_{n,d}$ 
a.a.s.\ have the corresponding property. A survey of properties of random 
$d$-regular graphs proved using this model is in~\cite{Wor99}.

We make isoperimetric definitions for pairings to coincide with the same 
parameters for the corresponding (multi)graphs. For a pairing 
$H\in{\cal P}_{n,d}$ and a subset $U$ of its $n$
cells, let $\partial_{V(H)}(U)$ denote the
set of all cells adjacent to a cell in $U$ but not in $U$. Similarly,
let $\partial_{E(H)}(U)$ denote the set of
all edges with exactly one endpoint in a cell of $U$. Note that 
$|\partial_{V(H)}(U)|\leq|\partial_{E(H)}(U)|$.
For any $0<u\le1$, the $u$-vertex isoperimetric number for
a pairing $H\in{\cal P}_{n,d}$ is defined as 
\[
i_{V,u}(H)=\min_{|U|\leq un}|\partial_{V(H)}(U)|/|U|,\]
and likewise for any $0<u\le1/2$ the $u$-edge isoperimetric
number is defined as \[
i_{E,u}(H)=\min_{|U|\leq un}|\partial_{E(H)}(U)|/|U|.\]
Furthermore, put
\[
i_{V,u}({\cal P},d)=\sup_{}\left\{ \ell:i_{V,u}({\cal P}_{n,d})\geq \ell\mbox{ a.a.s.}\right\} 
\]
and define $i_{E,u}({\cal P},d)$ analogously.

\section{Lower bounds for vertex expansion} \lab{s:vertparam}

For a sequence $u=u(n)$ with $0< u\le1$ for all $n$, we define the 
\emph{$u$-vertex expansion number} to be
\[
I_{V,u} (d)=\sup \left\{ \ell: \min_{U\subset V,\, |U|=un}
\frac{|\partial_{V}U|}{un} \geq \ell  \mbox{ a.a.s.\ in }{\cal G}_{n,d}\right\}.
\]

The motivation to study the vertex expansion number  
is the following relation to the isoperimetric number.

\begin{lem} \lab{l:param}
Fix $0< u_0\le 1$. Then 
\[
i_{V,u_0}(d)\ge \inf _{0\le u\le u_0}\inf _{w \to u} I_{V,w}(d),
\]
where the second infimum is over sequences  $w(n)$ with $0<w\le1$. \end{lem}

\begin{proof}
Set $L$ to be right-hand side of the inequality, and assume by 
way of contradiction that, for some $\eps>0$, $i_{V,u_0}(d)$ 
is not at least $L-\eps$ a.a.s. Then  for all $n$ in some infinite set 
$S$ of positive integers, and some $\eps'>0$, we have  
${\bf P}(i_{V,u_0}({\cal G}_{n,d}) <  L-\eps) >\eps'$. Thus, there exists a 
function $w(n)>0$ with  a limit point $u\in[0,u_0]$ such that  in ${\cal G}_{n,d}$
\[
{\bf P}\left(\min_{|U|=w(n)n}\frac{|\partial_V(U)|}{| U|}<L-\eps\right)>\eps'
\]
for all $n\in S$. 
For any sequence $w'(n)\to u$ with $w'(n)=w(n)$ for all $n\in S$, 
we have that $I_{V,w'}(d)\le L-\eps$,
giving the desired contradiction.\end{proof}

Of course we expect $I_{V,u}(d)\geq i_{V,u_0}(d)$ if $u\to u_0>0$  
since it seems natural to expect that large sets will cause the most problems. 
Equality will often hold but is not always straightforward to prove, so in some 
cases we will have to be satisfied with explicit results in the form  of a fixed 
continuous function $f$ such that $I_{V,u}\ge f(c)$ when 
$u(n)\sim c$. From our argument,  we will be able to conclude that for 
fixed $u$, $i_{V,u}\ge \min\{f(c):0\le c\le u\}$.

To analyse $I_{V,u} (d)$  it will be useful to first look at the analogously
defined quantity $I_{V,u}({\cal P},d)$ for pairings.

The main complication in bounding $I_{V,u}({\cal P},d)$
via the first moment method is that both $\partial_{V}$
and $\partial_{E}$ must be considered to compute the expected number
of elements of ${\cal P}_{n,d}$ with $i_{V,u}$ equal to some specified
value. Consequently, bounding $I_{V,u}$ is more involved than $I_{E,u}$, 
as in the latter case we need only take $\partial_{E}$ into account.

For a randomly selected element of ${\cal P}_{n,d}$, let
$X_{u,s,y,d}^{(n)}$ denote the number of subsets of $V$
of size $un$ that have $|\partial_{V}|=sn$ and $|\partial_{E}|=yn$. 
Here $s$, $u$ and $y$ will be functions of $n$.
Let $C_{n,s,y}$ denote the coefficient of $x^{yn}$ in the polynomial
\bel{poly}
\left(\sum_{j=1}^{d}{d \choose j}x^{j}\right)^{sn}=\big((x+1)^{d}-1\big)^{sn},
\ee
so that $C_{n,s,y}$ is the number of ways to choose $yn$ elements of 
$sdn$ items partitioned into $sn$ groups of cardinality $d$ each, such that 
at least one item is chosen from each group.
Note also that $M(2m) =(2m)!/m!2^{m}$ is the number of perfect
matchings of $2m$ points, so, for instance, $|{\cal P}_{n,d}| = M(dn)$. 
Then in ${\cal P}_{n,d}$, 
\[
{\bf E}\big(X_{u,s,y,d}^{(n)}\big)
=C_{n,s,y}\, {n \choose un}{n-un \choose sn}{dun \choose yn}
\frac{(yn)! M(dun-yn)M(dn-dun-yn)}
{M(dn)},
\]
where the binomial coefficients choose a set $U$ of $un$ vertices, 
their $sn$ neighbours, and the $yn$ points inside them that join to points 
outside $U $, and the other factors count choices of the pairs with the 
obvious restrictions. For any $x>0$,
\[
C_{n,s,y}\le x^{-yn}\big((x+1)^{d}-1\big)^{sn}.
\]
We will use this upper bound for various $x>0$ as our estimate for $C_{n,s,y}$. 

By Stirling's approximation, for any $x>0$, 
\[
\big({\bf E}\, X_{u,s,y,d}^{(n)}\big)^{1/n}
\le\frac{(du)^{du}(d-du-y)^{(d-du-y)/2}\big((x +1)^{d}-1\big)^{s}\phi(n)}
{x ^{y}u^{u}s^{s}(1-u-s)^{(1-u-s)}(du-y)^{(du-y)/2}d^{d/2}},
\]
where $\phi(n) = n^{O(1/n)}$ contains factors which are of 
polynomial size before taking the $n$th root.
Hence for $x>0$, we have
\bel{expected}
\log {\bf E} \big(X_{u,s,y,d}^{(n)}\big)  \le n\big(f_{d}(u,s,y,x) +o(1)\big)
\end{equation}
where
\begin{align*}
f_{d}(u,s,y,x)=& du\log(du)+(d-du-y)(\log(d-du-y))/2+s\log\big((x+1)^{d}-1\big)\\
&-y\log x -u\log u-s\log s -(1-u-s)\log(1-u-s)\\
&-(du-y)(\log(du-y))/2-(d\log d)/2.
\end{align*}
One particular value we will use is $x=x_0$, defined as the 
value at which $\partial f/\partial x=0$, or equivalently
\bel{xeqn}
\frac{sd(x+1)^{d-1}}{(x+1)^{d}-1 }=\frac{y}{x}.
\ee
This choice of $x$ is important, since if $(y-s)n\to\infty$ and $(sd-y)n\to\infty$ 
as $n\to\infty$ and $x_0>0$ is the unique solution to \eqn{xeqn}, 
then (for reasons explained below)
\bel{Casy}
x_{0}^{-y }\big((x_{0}+1)^{d}-1\big)^{s }\sim C_{n,s,y}^{1/n}\quad (\mbox{as } n\to\infty).\ee
Since the relevant range of $y$ is $s\le y\le ds$ (as for any $U\subset V$, 
$|\partial_V U|\le|\partial_E U|\le d\cdot|\partial_V U|$), 
it can be observed by \eqn{Casy} that using $x=x_0$ to bound $C_{n,s,y}$ 
in our argument leads to just as good final results   as  using the precise formula. 

We briefly outline two arguments for the asymptotics at \eqn{Casy}. 
One simple option is to use the limit theorems of 
Bender~\cite[Theorems 3 and 4]{Ben73} (cf.~\cite{OR85} and~\cite{DeA02}). 
Another more transparent way to obtain \eqn{Casy} is as follows: let $Y_i$ 
be i.i.d.\ random variables with support $\{0,1,\ldots,d-1\}$ and taking values 
$i$ with probability ${d\choose i+1}/(2^d-1)$.    
Put $Y(sn)=\sum_{i=1}^{sn}Y_i$, and observe
\[
C_{n,s,y}={\bf P}(Y(sn)=yn).
\]
Log-concave sequences are unimodal. Let $y^*n$ denote an exponent 
associated with the coefficient in the $sn$th convolution  of $p(x)$ attaining 
(as close as possible to) the centre of mass. By the Berry-Esseen inequality, 
$Y(sn)$ is asymptotically normal. Hence the asymptotics of $C_{n,s,y^*}$ 
can be established. Moreover, the asymptotics of an arbitrary coefficient, 
$C_{n,s,y}$, say, may be obtained as follows. For any $r\in\real$, we have
\[
r^m[x^m]p(x)^n=[x^m]p(xr)^n,
\]
where $[x^m]g(x)$ denotes the coefficient of $x^m$ in the polynomial $g$. 
Observe that since $p(x)$ has log-concave coefficients, so does $p(xr)$,   
and it then follows by a well-known property that the same holds for $p(xr)^n$. 
Hence, selecting $r$ such that the centre of mass in $p(xr)^n$ is located at the 
coefficient of $x^m$, the asymptotics of $[x^m]p(x)^n$ can be 
found by~\cite[Lemma~1]{Ben73}.

Note that $f$ is continuous on the natural domain in question, 
with the convention $0 \log 0=0$.
Define
\begin{align*}
M_d(u,s,y) &= \min_{x\ge 0} f_{d} (u,s,y,x),\\
h_d(u,s)&= \max_{s\le y\le \min\{ds,du\}} M_d(u,s,y)
\end{align*}
and
\[
H_d(u) = \min\big\{s:  h_d(u,s)\ge0 \big\}.
\]
A little examination of $f$ shows that the various $\min$'s and $\max$'s 
exist and are continuous. 
Recalling that every relevant $x$ leads to an upper bound in~\eqn{expected}, 
we may now deduce the following.
\begin{lem}\lab{maxf}
$H_d$  has the following properties.
\begin{itemize}
\item[(a)] $H_d(u)=0$ if and only if $u\in\{0,1\}$.
\item[(b)] Fix $0< u_0<1$. If   $u=u(n)\to u_0$ as $n\to\infty$, then
\[
I_{V,u}(d)\ge \frac{H_d(u_0)}{u_0}.
\]
In the case that $u\to0^+$,  
\[
I_{V,u}(d)\ge d-2.
\]
\item[(c)]
For any $0<u_0<1$, we have 
\[
i_{V,u_0}(d)\ge\inf_{0<u\le u_0} \frac{H_d(u)}{u}.
\]
\end{itemize}
\end{lem}
\begin{proof}
For part (a), note that $H_d(u)=0$ if and only if $u\in\{0,1\}$. 
Observe that $f_d(0,0,0,\cdot)=0=f_d(1,0,0,\cdot)$ (noting that when $s=0$, 
the only possible value $y$ in the $\max$ function in the definition of  
$h_d(u,s)$ is 0, and then $x$ does not appear in $f$). So $h_d(0,0)= h_d(1,0)\ge 0$, 
and hence $H_d(0)=H_d(1)=0$.  Conversely, suppose $H_d(u)=0$. Then 
$h_d(u,0)\ge0$, and so we have $f_d(u,0,0,\cdot)\ge0$. Thus, since 
\[
\frac{{\rm d}^2 f_d(u,0,0,\cdot)}{{\rm d}u^2}=\frac{d-2}{2u(1-u)}>0
\]
for $0<u<1$, $u$ is either 0 or 1.

For part (b), consider   $0<u_0<1$ and $u(n)\to u_0$. We have $H_d(u_0)>0$ by (a), 
so we can let $s_0$ be positive with $s_0< H_d(u_0)$.    Then $h_d(u_0,s_0)<0$.  
By  definition, $h_d(u,s)$ is nondecreasing in $s$. So, by  continuity,  
for $n$ sufficiently large, $h_d(u,s )<h_d(u_0,s_0)/2$ for all $s\le s_0$. 
Then~\eqn{expected} tells us that  in ${\cal P}_{n,d}$, the expected number 
of sets $U$ with $|U|=un$,    $|\partial_{V}U | =sn\le s_0n$, and  $|\partial_{E}U | =yn$   
is exponentially small for every relevant $y$  (noting that $y\ge s$ can be assumed 
because every boundary vertex has at least one boundary edge). Summing over all 
$O(n^2)$ relevant values of $s$ and $y$, we deduce by the union bound that a.a.s.\  
$|\partial_{V}U |\ge s_0n$ for all $U$ with $|U|=un$.  
Hence, $I_{V,u}({\cal P},d)\ge {H_d(u_0)}/{u_0}$.

Thus, the first inequality in (b) follows in view of the relation between 
${\cal G}_{n,d}$ and ${\cal P}_{n,d}$ discussed in Section~\ref{s:models}. 

In the case that $u\to0^+$, we use the small sets property discussed 
at \eqn{smallsets}: for all $\delta>0$, there is some $N_\delta\in{\mathbb N}$ so that for
$n>N_\delta$, $u(n)<\epsilon_\delta$ ($\epsilon_\delta$ as guaranteed by the property), 
and hence
\[
\min_{|U|=u(n)n}\frac{|\partial_V U|}{|U|}\ge 
\min_{|U|\le \epsilon_\delta n}\frac{|\partial_V U|}{|U|}\ge d-2-\delta
\]
a.a.s.\ in ${\cal G}_{n,d}$. The above holds for any $\delta>0$, so we have
\[
I_{V,u}(d)\ge d-2,
\]
which finishes the proof of (b).

For part (c), note that   Lemma~\ref{l:param} implies 
\[
i_{V,u_0}(d)
\ge \min\left\{\inf_{w\to 0}I_{V,w}(d),\inf_{0<u\le u_0} H_d(u)/n\right\}
\] 
and the first of these is at least $d-2$ by (b). On the other hand, by inequality \eqn{isoub}, 
$i_{V,u}(d)\le d-2$ for any $u>0$, and (c) follows. 
\end{proof}

Of course, this result is best possible for a direct application of the first 
moment method, in the sense that, from~\eqn{expected} and the earlier discussion, 
the expected number of sets of size $s$ with a boundary  which is slightly larger than 
$H_d(u_0)$ cannot be exponentially small. 

Now we need to discuss the behaviour of $f_{d} (u,s,y,x)$, which we often abbreviate to $f$.  
Similarly, since $d$ and $u$ are fixed for the whole discussion, we often refrain from explicitly 
mentioning them as parameters of other functions. 
 
A `direct attack,' solving for the minimum $x$ and maximum $y$ in the definition of $h_d(u,s)$, 
and then analysing its behaviour as a function of $s$, leads to calculations that seem too 
complicated. Instead we take an indirect approach: for each suitable $y$, we will compute a 
value of $s$ such that $y$ is the maximiser. Even with this, we do not restrict ourselves to 
using the minimising $x$, which, considering $\partial f/\partial x$, turns out to be $x_0$. 
For an upper bound, we are free to choose any $x$. To simplify the argument we will, 
for part of it, use a different choice of $x$, which happens to coincide with $x_0$ 
everywhere that matters. 

The relevant partial derivatives are 
\[
\frac{\partial f}{\partial x} 
=\frac{sd(x+1)^{d-1}}{(x+1)^d-1}-\frac{y}{x}
\]
and
\[
\frac{\partial f}{\partial y} = \log  \hat x(y)  - \log x,
\]
where
\[
\hat x(y) = \sqrt{\frac{du -y }{d-du -y}} 
\]
for all $y<du$.  

 For any $x$ and $y$, define 
 \bel{Seqn}
S(y,x)=\frac{y \big((x+1 )^{d}-1\big)}{xd (x+1)^{d-1} }    
\ee
so that  $s=S(y,x)$ satisfies~\eqn{xeqn}, and set 
\bel{shat}
\hat s(y)=S(y,\hat x(y)) 
\ee
and
\bel{Fdef}
 F( y )= F_{d,u}(y) = f_{d} (u,\hat s(y ),y ,\hat x(y ) ).
\ee

\begin{lem}\lab{l:shat}
Fix $0<u\le1/2$. We have ${\rm d}\hat s(y)/{\rm d}y >0$.
\end{lem}
\begin{proof}
This follows easily by checking that $\hat x(y)$ is a 
nonincreasing function of $y$ 
and that  $\partial S(x,y)/\partial x< 0$ (using $1+dx  -(x+1)^d<0$). 
\end{proof}

We are now ready to state our main result for this section. From now on, 
we restrict attention to $u\le 1/2$, since the result in Lemma~\ref{l:shat} 
does not always extend for much larger $u$. 
We say that a real-valued function $g$ with a  real domain $D$ is 
\emph{unimodal} with \emph{mode}
$\tilde y$ if $g(y)$ is strictly increasing for $y<\tilde y$ ($y\in D$) and 
strictly decreasing for $y>\tilde y$ ($y\in D$).

\begin{prop}\lab{p:mainvertex}
Let   $d\ge3$ and $0<u\le1/2$
and let $\hat s$ and $F$ be defined as in~\eqn{shat} and~\eqn{Fdef}. Then 
\begin{itemize}
\item[(a)] $F$ is unimodal with mode $\tilde y = du(1-u)$;
\item[(b)]  $F$ has a unique zero $\bar y\in(0,\tilde y)$, and we have that
\[
H_d(u)\ge \hat s(\bar y).
\]
\end{itemize}
\end{prop}

\begin{proof}
To analyse $F(y)$, we define
\[
g(s,y) = f_{d} (u,s,y,\hat x (y) )
\]
and note that by definition  $F(y) = g(\hat s(y ),y)$.
We have
\[
\frac{{\rm d} F}{{\rm d} y} = \frac{\partial g}{\partial y}\mid_{s=\hat s(y)} + 
\frac{\partial g}{\partial s}\mid_{s=\hat s(y)}\frac {{\rm d}\hat s}{{\rm d} y}.
\]
Now
\[\frac{\partial g}{\partial y}  
=  \frac{\partial f_{d} (u,s,y,x )}{\partial y}  \mid_{x=\hat x(y)}   
+ \frac{\partial f_{d} (u,s,y,x )}{\partial x} \mid_{x=\hat x(y)} 
\frac{{\rm d}\hat x}{{\rm d} y},
\]
where the first partial derivative is 0 by the way we defined   
$\hat x$, and the second one, evaluated at   $s=\hat s(y)$, is  
0 by~\eqn{shat} and the comment above it. Hence, the first partial derivative 
in the above formula for  ${\rm d} F /{\rm d} y $ is 0. Furthermore, 
by Lemma~\ref{l:shat}, ${\rm d}\hat s(y)/{\rm d}y >0 $. Thus, ${\rm d} F/{\rm d} y$ 
has the same sign as 
\[
\frac{\partial g}{\partial s}\mid_{s=\hat s(y)} 
=\frac{\partial f_{d}}{\partial s}  (u,\hat s(y),y,\hat x (y) ) 
= -\log \hat s(y) + \log \big(1-u-\hat s(y)\big) +\log\big((\hat x(y) +1)^d-1\big).
\]
Put $\tilde y=du(1-u)$. Observe that
\bel{x-tilde}
\hat x(\tilde y)=\frac{u}{1-u}
\ee
and hence
\bel{s-tilde}
\hat s(\tilde y)= S(\tilde y,\hat x(\tilde y))=(1-u)(1-(1-u)^d).
\ee
Thus, after some simplifications, we find that
\[
\frac{\partial g}{\partial s}\mid_{s=\hat s(\tilde y)}=0.
\]

 The partial derivative of $-\log  s  + \log \big(1-u- s \big)$ with 
 respect to $s$ is negative, $ {\rm d}\hat s(y)/{\rm d}y >0 $, and 
 ${\rm d}\hat x(y)/{\rm d}y < 0 $  as can be verified directly.   
 So $\frac{\partial g}{\partial s}\mid_{s=\hat s(y)}$, and consequently 
 also ${\rm d}F(y)/{\rm d}y$, 
takes the value 0 at $y=\tilde y$, and it is  positive for $y<\tilde y$ 
and negative for $y>\tilde y$. This gives the unimodality claim in (a), 
 so we may turn to (b).
 
We first address the existence of $\bar y$. Performing some 
straightforward manipulations we see that
\[
\lim_{y\to0^+}F(y) = \frac{d-2}{2}(u\log u +(1-u)\log(1-u))<0.
\]
Then using the simplified expressions \eqn{x-tilde} and \eqn{s-tilde}, 
we observe that 
\[
F(\tilde y)=-u\log u-(1-u)\log(1-u)>0.
\]
Hence, there is a point $\bar y\in (0,\tilde y)$ which satisfies $F(\bar y)=0$.
By the unimodality of $F$, it is unique.

 We next  show that, essentially, when computing $H_d$, for the maximum 
 in the definition of $h_d$, we can restrict ourselves to points $(s,y)$ 
 of the form $(\hat s(y ),y)$. 
 Fix $y_0<du$. We claim that for all
$y<du$,   
\bel{maxM}
M_d(u,\hat s(y_0 ),y)\le F( y_0 )
= f_{d} (u,\hat s(y_0 ),y_0 ,\hat x(y_0 ) ).
\ee
To see this, note that 
\[
 M_d(u,\hat s(y_0 ),y)\le  f_{d} (u,\hat s(y_0 ),y  ,\hat x(y_0 ) )
\]
and from above
\[
\partial  f_{d} (u,\hat s(y_0 ),y  ,\hat x(y_0 ) )/\partial y 
= \log  \hat x(y)  - \log \hat x(y_0 ).
\]  
Since ${\rm d} \hat x(y)^2 / {\rm d}  y = d(2u-1)/(d-du-y)^2$, 
$ \hat x(y)$ is a nonincreasing function of $y$. This implies  
that for fixed $y_0$, $ f_{d} (u,\hat s(y_0 ),y  ,\hat x(y_0 ) )$ 
is maximised at $y=y_0$, which gives~\eqn{maxM}. From this, it follows  that
\bel{Fsuffices}
h_d\big(u,\hat s(y_0 )\big) 
=\max_{\hat s(y_0 )\le y\le d\hat s(y_0 )}M_d(u,\hat s(y_0 ),y)\le F(y_0).
\ee
Recalling  that $\hat s$ is a continuous increasing function of $y$ 
which tends to 0 from above with $y$, it follows that if $s<\hat s(\bar y)$, 
then $s=\hat s(  y)$ for some $y<\bar y$ and
\bel{sto0}
 h_d(u,s) = h_d\big(u,\hat s(y )\big)\le  F(y)< F(\bar y)=0,
 \ee
where the last inequality follows since ${\rm d}F/{\rm d}y > 0$ on $(0,\bar y)$.
Thus $H_d(u)\ge \hat s(\bar y)$,  as required for (b).
\end{proof}

Since $\bar y$ as in Proposition~\ref{p:mainvertex}(b) depends on $d$ and $u$, 
we denote it by $\bar y_{d,u}$. To be clear, let us emphasize that 
$F(\bar y_{d,u})=0$ means
\[
f_d(u,\hat s(\bar y_{d,u}),\bar y_{d,u},\hat x(\bar y_{d,u}))=0.
\]

Next, we define
\bel{Adu}
A_d(u) = \hat s(\bar y_{d,u})/u.
\ee

Combining the proposition with Lemma~\ref{maxf}, 
we obtain the following immediately. 

\begin{cor}\lab{C_Adu}
If $u\to u_0$ where $0<u_0\le 1/2$ is fixed, then
\[
I_{V,u}(d)\ge A_d(u_0).
\]
\end{cor}
 
\noindent By this corollary and Lemma~\ref{maxf}(c), for any $u_0\in(0,1/2]$, we have 
\bel{I-Adu}
i_{V,u_0}(d)\ge\inf_{0<u\le u_0}A_d(u).
\ee

Approximate values of $A_d(u)$ for various $d$ and $u$ are 
provided in Table~\ref{tablev}. These were found by searching for the first zero 
of $F=F_{d,u}(y)$ and finding strictly positive and negative values of $F$ on either 
side of it. 
Recall that finding such a zero of $F$ as a function of $y$ means we have found 
$\bar y_{d,u}$ and hence $A_d(u)$ via \eqn{Adu}. The entires in the table are 
monotonically decreasing in the columns, and this seems likely to hold for all $d$. 
If this is true, it would follow from \eqn{I-Adu} that  $i_{V,u}(d)\ge A_d(u)$ in general.
 
\begin{table}[!h]\footnotesize
\centering
\caption{Approximate values for $A_d(u)$. By Corollary~\ref{C_Adu}, these 
are approximate lower bounds for the $u$-vertex expansion number $I_{V,u}(d)$.}
\vspace{-0.2cm}
\begin{tabular}{c|ccccccc}
\toprule
$u$ & $\approx A_3(u)$ & $\approx A_4(u)$ & $\approx A_5(u)$ & $\approx A_{10}(u)$ & 
$\approx A_{25}(u)$ & $\approx A_{50}(u)$  & $\approx A_{100}(u)$  \\
\hline
$0.01$ &0.55822 & 1.24636 & 1.97397 & 5.71086 &16.16640 & 30.80253 & 52.21931 \\
$0.05$ & 0.43552 & 0.97129 & 1.52478 & 4.12128 &9.57894 & 14.12199 & 17.14034 \\
$0.10$ & 0.36513 & 0.80589 & 1.24807 & 3.13558 &6.15315 & 7.78467 & 8.52607 \\
$0.15$ & 0.31790 & 0.69369 & 1.06039 & 2.50085 &4.35286 & 5.11942 & 5.43785 \\
$0.20$ & 0.28136 & 0.60687 & 0.91620 & 2.04298 &3.25720 & 3.68551 & 3.86267 \\
$0.25$ & 0.25110 & 0.53536 & 0.79862 & 1.69322 &2.52784 &  2.79584 & 2.90837 \\
$0.30$ & 0.22503 & 0.47421 & 0.69923 & 1.41621 &2.01058 & 2.19121 & 2.26836 \\
$0.35$ & 0.20194 & 0.42060 & 0.61319 & 1.19112 &1.62589 &  1.75381 & 1.80926 \\
$0.40$ & 0.18108 & 0.37272 & 0.53737 & 1.00461 &1.32904 &  1.42271 & 1.46383 \\
$0.45$ & 0.16196 & 0.32936 & 0.46968 & 0.84761 &1.09323 & 1.16336 & 1.19447 \\
$0.50$ & 0.14420 & 0.28966 & 0.40859 & 0.71371 &0.90142 &  0.95467 & 0.97850 \\
\bottomrule
\end{tabular}
\label{tablev}
\end{table}

\section{Lower bounds for the vertex isoperimetric number} \lab{s:uhalf}

For $i_{V,u_0}$, as opposed to $I_{V,u_0}$ (for a sequence $u(n)\to u_0$), 
we must consider vertex sets of cardinality {\em less than or} equal to $u_0n$. 
As noted in \eqn{I-Adu}, the minimum of $A_{d}(u)$ over all $0<u\le u_0$ gives 
a lower bound for $i_{V,u_0}$. However, this turns out to be not so amenable to 
theoretical analysis. 

In this section, for the case $u_0=1/2$, we use a less direct argument to 
show that $i_{V}(d)\ge   A_{d}(1/2)$. This bound is best possible for a direct 
application of the first moment method, since it was best possible for
$I_{V,u}(d)$ over all $u\to1/2$. (See the comment after Lemma~\ref{maxf}.)

The situation is manageable at $u_0=1/2$  
largely due to the fact that in this case
\[
\hat{x}(y)=\sqrt{\frac{d/2-y}{d-d/2-y}}=1\]
for all relevant $y$, and we get the simplified expression  
\bel{new}
\hat{s}(y)=S(y,1)=2y(1-1/2^{d})/d,
\ee
which is useful for computing $F(y)$ for $u_0=1/2$, and hence $A_{d}(1/2)$. 
Corollary~\ref{C_Adu} provides us with the lower bound $I_{V,u}(d)\ge A_{d}(1/2)$ 
when $u\to1/2$. However, to get a lower bound on $i_{V}(d)$ with Lemma~\ref{maxf}, 
we need to also consider $H_d(u)/u$ for $0< u< 1/2$. For such $u$, we use 
inequality~\eqn{expected}  to show that the case $u_0=1/2$   is critical, 
in the sense that $A_{d}(1/2)\le I_{V,u}(d)$  for any sequence $u(n)\to u\in(0,1/2]$.

Let $\hat {\cal A}$ be the set of all $(u,s)$ with $s/u< A_{d}(1/2)$ and $0<u\le1/2$. 
We will show that $h_d(u,s)<0$ on $\hat {\cal A}$. It then follows that $H_d(u)>s$ 
for all $(u,s)\in \hat {\cal A}$, and then, by Lemma~\ref{maxf}, we may 
conclude $i_{V}(d)\ge A_{d}(1/2)$.

To this end, define
\bel{hhat}
\widehat h_d(u,s)= \max_{s\le y\le \min\{ds,du\}} f_{d} (u,s,y,1).
\ee
Since 
$ f_{d} (u,s,y,1)\ge M_d(u,s,y)$, we have
$  \widehat h_d(u,s) \ge  h_d(u,s) $, and it suffices to show 
$ \widehat h_d(u,s)<0$ on $\hat {\cal A}$.   
Our job will be made easier after we show that we have
$\partial f_{d} (u,s,y,1) /\partial y<0$.  

First, since we do not have a closed form for $A_{d}(1/2)$, 
we obtain an initial estimate of the location of 
$\hat {\cal A}$ with the following lemma.
Define 
\[
C(d)=(d-2 )/ (d-1).
\]

\begin{lem}\lab{l:ibound}
For any $d\geq3$, we have $0<A_{d}(1/2)< C(d)$, 
and for all $s<A_d(1/2)/2$,  $\widehat h_d(1/2,s)<0$.
\end{lem}
    
\begin{proof}
Fix $d$. Put $y_{d}=d (d-2 )2^{d-2}/ (d-1 ) (2^{d}-1 )$. 
(We follow the convention $a/bc = a/(bc)$, that is, 
multiplication by juxtaposition takes precedence over `$/$'.)
We have 
\begin{align*}
F_{d,1/2}(y_{d})&=-\big((d^{2}-3d+2)\log 2 +\log(1/(d-1))\big)/2(d-1)\\
&\quad +(d-2)\log\big((d-1)(2^{d}-1)/(d-2)\big)/2(d-1)\\
&>-\big((d^{2}-3d+2)\log 2 -(d-2)\log(2^{d}-1)\big)/2(d-1)\\
&=(d-2)(\log(2^{d}-1)-(d-1)\log 2 )/2(d-1)\\
&=(d-2)\log ((2^{d } -1)/2^{d-1})/2(d-1)\\
&>0.
\end{align*}

Hence $A_d(1/2)=2\hat s(\bar y_{d,1/2})$,
where $0<\bar y_{d,1/2}<y_d$. Since $\hat s$ is 
nonnegative and monotonically increasing, we have by \eqn{new} that
\[
0<A_d(1/2)<2\hat s(y_d)=4y_d(1-1/2^d)/d = (d-2)/(d-1)=C(d).
\]
This establishes the first inequalities in the lemma.
 
Since $\hat x(y)\equiv1$ at $u=1/2$ and so $\widehat h_d(1/2,s)=h_d(1/2,s)$, 
we have  $\widehat h_d(1/2,s)<0$ for all $s<A_d(1/2)/2$  by (\ref{sto0}), 
where, in this instance, $\bar y =\bar y_{d,1/2}$.
\end{proof}

We are ready to prove the main result for this section.

\begin{proof}[Proof of Theorem~\ref{t:mainvertexhalf}]
Fix $d\ge 3$. It will be convenient to parameterize $s$ in terms of $u$ 
and a new variable $r$. Set $s=ru$ and note that if $r=0$, then necessarily $y=0$ 
(since $\partial_V=0$ if and only if $\partial_E=0$). Recall that, by the discussion 
just before Lemma~\ref{l:ibound}, it is enough to show that 
$ \widehat h_d(u,s)<0$ on $\hat {\cal A}$. Observe that for relevant $s$ and $u$, 
\[
\frac{{\rm d}f_{d}(u,s,y,1)}{{\rm d}y}=\log\sqrt{\frac{du-y}{d-du-y}}<0.
\]
Hence $\widehat h_d(u,ru)= f_d(u,ru,ru,1)$. Moreover, the final inequality in 
Lemma~\ref{l:ibound} takes care of the case $u=1/2$. Thus, it suffices to show that  
\[
g_{d}(u,r):=f_d(u,ru,ru,1)<0
\] 
on ${\cal A}$, where
\[
{\cal A}=\left\{ (u,r) : 0<u<1/2, 0\le r< A_d(1/2)\right\}.
\]
For all $0\leq r<C(d)$,  it is easy to check that
\bel{gat0}
\lim_{u\to0^{+}}g_{d}(u,r)=0.
\ee
Again, by the final inequality in Lemma~\ref{l:ibound}, $g_d(1/2,r)<0$ 
for all $0<r<A_d(1/2)$. Hence it suffices to show that for each $0\le r<C(d)$, 
$g_{d}(u,r)$ is either strictly convex or strictly increasing in $u$ at every $0<u<1/2$.  
We partition the region $\cal A$  as follows, with a view to showing that 
$g_{d}(u,r)$ is either increasing or convex in each region. Define 
\begin{align*}
{\cal A}_{0}&=\left\{ (u,r): r=0,\ 0<u<1/2 \right\},\\
{\cal A}_{1}&=\left\{ (u,r): 0<r\le \min\{c(d),A_d(1/2)\},\ 0<u<1/2\right\},\\
{\cal A}_{2}&=\left\{ (u,r): c(d)<r<A_d(1/2),\ 0<u<U(r,d)\right\}, \\
{\cal A}_{3}&=\left\{ (u,r): c(d)<r<A_d(1/2),\ U(r,d)\le u<1/2\right\},
\end{align*}
with
\[
U(r,d)=\min\bigg\{1/2,\frac{d(d-2-r)}{(r+1)d^{2}+(d-2)r^{2}-2r-d(r+2)}\bigg\} 
\]
and $c(d)=(d-2)/(d+2)$.

We analyse $g_{d}$ in the domain corresponding to each portion of ${\cal A}$, 
starting with ${\cal A}_{0}$. Since 
\[
\frac{{\rm d}^{2}g_{d}(u,0)}{{\rm d}u^{2}}=(d-2)/2u(1-u)>0,
\]
$g_{d}(u,r)$ is strictly convex in $u$ for $0<u<1/2$, as required for ${\cal A}_{0}$.

Assume hereafter $r>0$. Note that
\[
\frac{{\rm d}^{2}g_{d}(u,r)}{{\rm d}u^{2}}=\frac{\zeta(r,u,d)}{\eta(r,u,d)},
\]
with
\begin{align*}
\zeta(r,u,d)&=(1-u-ru)d^{2}+(-2+ru-r^{2}u+2u-r)d+2ru(r+1),\\
\eta(r,u,d)&=2u( 1-u-ru)( d-du-ru).
\end{align*}
As $u\leq1/2$ and $r\le A_d(1/2)<C(d)<1$, we have $\eta(r,u,d)>0$.
Further, \[
\frac{{\rm d}\zeta(r,u,d)}{{\rm d}u}=-(1+r)(d-2)(d+r)<0.\]
Hence to show ${\rm d}^{2}g_{d}(u,r)/{\rm d}u^{2}>0$, 
it is enough to determine that
\[
\zeta(r,1/2,d)=(1-d/2)r^{2}+(1-d/2-d^{2}/2)r+d(d/2-1)>0.
\]
The coefficient of $r$ and $r^{2}$ above are both negative since
$d\ge3$. Therefore we have that $\zeta(r,1/2,d)>0$ in ${\cal A}_{1}$
since \[
\zeta(c(d),1/2,d)=4d(d-2)/(d+2)^{2}>0
\]
for $d\ge3$. So $g_{d}(u,r)$ is strictly convex in $u$ for fixed 
$r$ such that $(u,r)\in{\cal A}_{1}$, as required. 
 
As seen above,  $\zeta(r,u,d)$ is decreasing in $u$. 
Observe that whenever $U(r,d)\le1/2$,
$\zeta(r,U(r,d),d)=0$. Thus for any $r\in(c(d),C(d))$,
$g_{d}(u,r)$ is strictly convex in $u$ for $0<u<U(r,d)$.
That is, $g_{d}(u,r)$ is strictly convex in $u$ for fixed $r$ 
such that $(u,r)\in{\cal A}_{2}$. On the other hand, if 
$r\in(c(d),C(d))$ and $U(r,d)\le u<1/2$, then
the situation is as follows. By the definition of $U(r,d)$ and the 
fact that $\zeta$ is linear in $u$, we deduce that ${\rm d}g_{d}(u,r)/{\rm d}u$ 
is decreasing in $u$. Hence, to show that $g_{d}(u,r)$ is increasing in $u$, 
it suffices to show that
\begin{align*}
\frac{{\rm d}g_{d}(u,r)}{{\rm d}u}\big|_{u=1/2}
&=r\log((1-r)/r)+r\log(2^d-1) +d\log(d/(d-r)) +\log(1-r)\\
&=\log\left({\frac{(1-r)^{1+r}(2^d-1)^rd^d}{r^r(d-r)^d}}\right)\\
&>0
\end{align*}
for $r\in[c(d),C(d)]$.
To see that the above holds for $d\ge8$, observe that for this range of $d$ we have
\[
r(1-r)^{-\frac{1+r}{r}}
<C(d)(1-C(d))^{-\frac{1+C(d)}{c(d)}}
<(2^d-1)\left(\frac{d}{d-c(d)}\right)^{\frac{d}{C(d)}}
<(2^d-1)\left(\frac{d}{d-r}\right)^{\frac{d}{r}}.
\]
For $d\le7$, put $\theta_d(r )={\rm d}g_d(u,r)/{\rm d}u\mid_{u=1/2}$. We have
\[
\frac{{\rm d}^{2}\theta_d(r )}{{\rm d}r^{2}}=-\frac 2{1-r}-\frac{1+r}{(1-r)^2} -\frac1{r}+\frac{d}{(d-r)^2}<0
\]
for $d\ge 3$  and $r <1$, since each of the first three terms is less than 
$-1$ and the last is less than 1. By direct calculation,  all the endpoints   
$ \theta_d(c(d))$ and  $\theta_d(C(d))$ are positive for $3\le d\le 7$. 
So by the concavity of $\theta_d(r)$ in $r$, we have that $\theta_d(r)>0$ for all 
$d\le7$ and relevant $r$. Therefore $g_{d}(u,r)$ is strictly increasing in $u$ for 
fixed $r$ such that $(u,r)\in{\cal A}_{3}$. Putting the above together, we 
conclude that $g_{d}<0$ on ${\cal A}_{2}\cup{\cal A}_{3}$. 

Altogether we have shown $g_{d}<0$ on ${\cal A}=\bigcup_{i=0}^{3}{\cal A}_{i}$. 
As noted earlier, this implies $i_{V}(d)\ge  A_d(1/2)$.
\end{proof}

To supplement the values in Table~\ref{tablev}, additional
approximations to 
$A_d(1/2)$ for various $d$ were generated by the same method, 
as shown in Table~\ref{tablevhalf}.

\begin{table}[!h]\footnotesize
\centering
\caption{Approximate values for $A_d=A_d(1/2)$. By Theorem~\ref{t:mainvertexhalf}, 
these are approximate lower bounds for the vertex isoperimetric number $i_V(d)$.}
\vspace{-0.2cm}
\begin{tabular}{cc|cc|cc|cc|cc}
\toprule
$d$ & $\approx A_{d}$ & $d$ & $\approx A_{d}$ & $d$ & $\approx A_{d}$ & $d$ 
& $\approx A_{d}$ & $d$ & $\approx A_{d}$ \\
\hline
3 &0.14420 &11 &0.74355 &19 &0.86463 &27 &0.90972 &35 &0.93269\\
4 &0.28966 &12 &0.76827 &20 &0.87246 &28 &0.91338 &40 &0.94201\\
5 &0.40859 &13 &0.78897 &21 &0.87948 &29 &0.91677 &50 &0.95467\\
6 &0.50190 &14 &0.80652 &22 &0.88579 &30 &0.91991 &60 &0.96284\\
7 &0.57466 &15 &0.82155 &23 &0.89149 &31 &0.92283 &70 &0.96854\\
8 &0.63178 &16 &0.83455 &24 &0.89668 &32 &0.92555 &80 &0.97274\\
9 &0.67716 &17 &0.84587 &25 &0.90141 &33 &0.92809 &90 &0.97596\\
10 &0.71371 &18	&0.85582 &26 &0.90574 &34 &0.93046 &100 &0.97850\\
\bottomrule
\end{tabular}
\label{tablevhalf}
\end{table}

\section{Asymptotic bounds for $i_{V}(d)$}\lab{s:asymp}

The asymptotics of $A_{d}(1/2)$, as $d\to\infty$, can be computed as
follows. For the case $u=1/2$, we have, as stated 
above at \eqn{new}, $\hat x(y)=1$ and
\[
\hat s(y) = 2y(1-1/2^d)/d.
\]
Hence, as discussed after the proof of 
Proposition~\ref{p:mainvertex},
\[
A_d(1/2)=4\bar y(1-1/2^d)/d,
\]
where $\bar y$ uniquely satisfies
\[
f_d(1/2,2\bar y(1-1/2^d)/d,\bar y, 1)=0.
\]
However, observe that
\[
f_{d}(1/2,s,y,1)=
s\log(2^d-1)+(\log2)/2-s\log s-(1/2-s)\log(1/2-s)-(d\log 2)/2.
\]
Hence, when $u=1/2$ and $x=1$, $y$ does not appear in $f_d$. 
Thus, we investigate the asymptotics of $s_d=2\bar y(1-1/2^d)/d$ 
satisfying
\[
f_d(1/2,s_d,\cdot,1)=0.
\]
Moreover, note that by Lemma~\ref{l:ibound}, $A_d(1/2)>0$ for all $d\ge3$. 
Hence, for all $d\ge 3$, $s_d>0$. In fact, we can show that $s_d\to1/2$. 
Writing $\log (2^d-1)=d\log 2 +\log (1-1/2^d)$ and manipulating,   we obtain
\begin{align*}
f_{d}(1/2,s,\cdot,1) 
=& s\big(-\log s-\log 2+\log(1-2s)+d\log 2+\log(1-1/2^d)\big)-(d\log 2)/2\\
  &+\log 2-\log(1-2s)/2\\
=&(s-1/2)d\log 2 -s\log s +(s-1/2)\log (1-2s)+(1-s)\log 2\\ 
  & +O(s/2^d)\\
=&(s-1/2)d\log 2 +O(1).\end{align*}
Setting this equal to 0, we conclude $s_d= 1/2 +O(1/d)$ as $d\to\infty$. 
With this in mind, we make the change of variables $t=1/2-s_d$ in the 
above expression, obtaining 
\[
0=-td\log2-(1/2) \log(1/2)+O(1/d)+O(\log d/d)+(1/2)\log 2
\]
and hence
\[
t=1/d+O(\log d/d^{2}).
\] 
Therefore 
\bel{asymp2}
A_{d}(1/2)=1-2/d+O(\log d/d^{2}).\ee
 By Theorem~\ref{t:mainvertexhalf} and (\ref{asymp2}), we deduce 
 Corollary~\ref{t:1/2asymp}.
Observe that Corollary~\ref{t:1/2asymp} provides a stronger bound 
on the asymptotics of $i_{V}(d)$ as $d\to\infty$ than~\eqn{asymp1}.

\section{A note on the edge isoperimetric number}\lab{s:edge}

 In this section we show how the above arguments can be modified to 
 obtain a.a.s.\ lower bounds for the edge isoperimetric number of 
 random regular graphs. As discussed in Section~\ref{s:intro}, 
 Bollob\'as~\cite{Bol88} computed lower bounds on $i_{E}(d)$ for all $d\ge3$ 
 with the first moment method. Also, the asymptotics of $i_{E}(d)$, as $d\to\infty$, 
 are investigated. Moreover, it is claimed that the arguments could be 
 modified for $0<u<1/2$.  
However, explicit lower bounds on $i_{E,u}(d)$ are not given, nor are 
they given for the asymptotics of $i_{E,u}(d)$, 
as $d\to\infty$, for the cases $0<u<1/2$. 
In this section we provide lower bounds on $i_{E,u}(d)$ which result 
from direct application of the first moment method 
for all $d\ge3$ and $0<u\le1/2$. The bounds are analysed 
asymptotically as $d\to\infty$ for fixed $u$.

For a randomly selected element of ${\cal P}_{n,d}$, let $X_{u,y,d}^{(n)}$
denote the number of subsets of $V$ of size $un$ that have $|\partial_{E}|=yn$.
Then 
\[
{\bf E}\left(X_{u,y,d}^{(n)}\right)
={n \choose un}{dun \choose yn}{dn-dun \choose yn}
\frac{(yn)!M(dun-yn)M(dn-dun-yn)}{M(dn)},
\]
where $M(2m)=(2m)!/m!2^{m}$ counts the number
of perfect matchings of $2m$ points, the binomial coefficients choose
a set $U$ consisting of $un$ vertices and $yn$ boundary edges,
and the other factors count choices of the pairs with the obvious
restrictions. Therefore, via Stirling's approximation,
\[
\left({\bf E}\, X_{u,y,d}^{(n)}\right)^{1/n}
=\frac{(du)^{du}(d-du)^{d-du}\phi(n)}
{u^{u}(1-u)^{1-u}y^{y}(du-y)^{(du-y)/2}(d-du-y)^{(d-du-y)/2}d^{d/2}},
\]
 where $\phi(n)=n^{O(1/n)}$ contains the factors
of polynomial size before taking the $n$th root. Hence 
\bel{edgeexp}
\log{\bf E}\left(X_{u,y,d}^{(n)}\right)\le n(\widehat f_d(u,y)+o(1)),
\ee
where
\begin{align*}
\widehat f_d(u,y)=&du\log(du)+(d-du)\log(d-du)-u\log{u}-(1-u)\log(1-u)-y\log{y}\\
&-\big((du-y)\log(du-y)+(d-du-y)\log(d-du-y)+d\log{d}\big)/2.
\end{align*}
Up to this point, these facts are essentially contained in~\cite{Bol88}.
To get lower bounds on $i_{E,u}(d)$ we find where $\widehat f_d<0$ 
and use \eqn{edgeexp}. The argument is comparable to that of the 
preceding sections; however, the situation is much simpler since in the 
current case we analyse a function with only two parameters. 
(Recall that in Section~\ref{s:vertparam}, there was a function $f_d$, 
defined at \eqn{expected}, with parameters corresponding to the edge 
and vertex boundary sizes and also one used to estimate a 
polynomial coefficient.)

Let us collect some facts about  $\widehat f_d$.
Note that if $G\in{\cal G}_{n,d}$ and $|U|=un$ with 
$u\le 1/2$, then $|\partial_E U|\le dun$.

Fix $d\ge3$ and $0<u\le1/2$. Note that $\widehat f_d(u,y)$ is 
strictly concave in $y$ for $0\le y<du$. Indeed, so long as 
$0\le y< du$, we have
\[
\frac{{\rm d}^{2}\widehat f_d(u,y)}{{\rm d}y^{2}}
=\frac{d(y-2du(1-u))}{2y(du-y)(d-du-y)}<0.
\]
We define 
\[
\widehat  A_d(u)=\frac{1}{u}\min\{y:\widehat f_d(u,y)\ge0\}.
\]
To see that $0< \widehat A_d(u)\le d(1-u)$, observe that
\[
\lim_{y\to0^{+}}\widehat f_d(u,y)=\frac{d-2}{2}(u\log u+(1-u)\log(1-u))<0
\]
and that, after some simplifications, 
\[
\widehat f_d(u,du(1-u))=-u\log u-(1-u)\log(1-u)>0.
\]

As we did for the case of vertex expansion, we will define a pointwise 
measure of edge expansion. For a sequence $u=u(n)$ with 
$0<u\le1/2$ for all $n$, we define the \emph{$u$-edge expansion number}  
to be 
\[
I_{E,u} (d)
=\sup \left\{ \ell: \min_{U\subset V,\, |U|=un}\frac{|\partial_{E}U|}{un} 
\geq \ell  \mbox{ a.a.s.\ in }{\cal G}_{n,d}\right\}.
\]

We state here analogues of Lemmas~\ref{l:param} and~\ref{maxf} 
for the case of edge expansion. 
We do not provide the proofs since they are very similar.

\begin{lem}\lab{l:eparam}
Fix $0<u_0\le1/2$. Then
\[
i_{E,u_0}(d)\ge  \inf_{0\le u\le u_0}\inf_{w\to u}I_{E,w}(d),
\]
where the second infimum is over sequences $w(n)$ with $0<w\le1/2$.
\end{lem}
\begin{proof}
The proof is analogous to that of Lemma~\ref{l:param}.
\end{proof}

\begin{lem}\lab{l:emaxf}
$\widehat A_d$ has the following properties.
\begin{itemize}
\item[(a)] Fix $0< u_0\le 1/2$. If   $u=u(n)\to u_0$ as $n\to\infty$, then 
\[
I_{E,u}(d)\ge \widehat  A_d(u).
\]
In the case that $u\to0^+$,  
\[
I_{E,u}(d)\ge d-2.
\]
\item[(b)]
For any $0<u_0\le1/2$, we have 
\[
i_{E,u_0}(d)\ge\inf_{0<u\le u_0} \widehat  A_d(u).
\]
\end{itemize}
\end{lem}
\begin{proof}
The proof is analogous to that of Lemma~\ref{maxf}.  
(Note that the small sets property, as discussed at 
\eqn{smallsets} and \eqn{isoub}, also holds for edge expansion.)
\end{proof}

As we now prove, $\widehat A_d(u)$ is, in fact, a lower bound  
for $i_{E,u}(d)$ for all $d\ge3$ and $0<u\le1/2$.

\begin{proof}[Proof of Theorem~\ref{t:eThm}]
Fix $d\ge3$ and $0<u\le1/2$. As it was in the proof of 
Proposition~\ref{p:mainvertex}, we will consider $0<w\le u$ 
and  parameterize the variable $y$ as $rw$, where $r$ is a 
new variable. Put 
\[
{\cal B}=\left\{(r,w):0\le r< \widehat  A_d(u),0<w\le u\right\}.
\]
Once we show $g_d(r,w)=\widehat f_d(w,rw)<0$ over ${\cal B}$, 
the theorem will follow by 
applying 
Lemma~\ref{l:emaxf} and 
inequality \eqn{edgeexp}. Partition ${\cal B}$ as follows:
\begin{align*}
{\cal B}_1&={\cal B}\cap 
\left\{(r,w):w\leq W_{r,d}\mbox{  or }r\leq R_{w,d}\right\},\\
{\cal B}_2&={\cal B}\setminus {\cal B}_1,
\end{align*}
where 
\begin{align*}
W_{r,d}&=\frac{d(d-2-r)}{(d-2)(d+r)},\\
R_{w,d}&=\frac{d(d-2)(1-w)}{d+(d-2)w}.
\end{align*}
We have 
\[
\frac{{\rm d}^{2}g_{d}(r,w)}{{\rm d}w^{2}}=\frac{\eta(d,r,w)}{\zeta(d,r,w)},
\]
where
\begin{align*}
\eta(d,r,w)&=(1-w)d^2-(2(1-w)+r(1+w))d+2rw,\\
\zeta(d,r,w)&=2w(d-dw-rw)(1-w).
\end{align*}
As $w\leq u\leq1/2$ and $r<du$, $\zeta(d,r,w)>0$. Also, we have
\[
\eta(d,r,W_{r,d})=\eta(d,R_{w,d},w)=0,
\]
and if $w<W_{r,d}$ or $r<R_{w,d}$, then $\eta(d,r,w)>0$. 
Hence, over ${\cal B}_1$, $g_d(r,w)$ is convex in $w$ for any fixed $r$.

Regarding ${\cal B}_2$, we know that 
\[
\frac{{\rm d}^{2}g_d(r,w)}{{\rm d}w^{2}}<0
\] 
and so ${\rm d}g_{d}(r,w)/{\rm d}w$ is decreasing in $w$. 
Thus, to show that $g_d(r,w)$ for a fixed $r$ is increasing in $w$  
over ${\cal B}_2$, it suffices to show that 
\[
\theta_{u}(r)=\frac{{\rm d}g_{d}(r,w)}{{\rm d}w}|_{w=u}>0
\] 
for all $R_{u,d}\le r\le d(1-u)$, since $R_{w,d}$ is decreasing in 
$w$ and $\widehat A_{d,u}<d(1-u)$. 
First, observe that
\[
\frac{{\rm d}^{2}\theta_{u}(r)}{{\rm d}r^{2}}
=\frac{d^2(1-2u^2)r-2d^3(1-2u+u^2)}{2r(d-r)(d-du-ru)^2}<0
\]
for $0<r<2d(1-2u+u^2)/(1-2u^2)$. 
So as
\[
\frac{2d(1-2u+u^2)}{(1-2u^2)}-d(1-u)=\frac{d(1-u)(2u^2-2u+1)}{1-2u^2}>0,
\]
$\theta_{u}$ is strictly concave in $r$ over the interval in question. 
Thus we check that $\theta_{u}$ is positive at the endpoints. 
The right endpoint is positive since, after some simple manipulations, we see that
\[
\frac{{\rm d}\theta_{u}(d(1-u))}{{\rm d}u}=-\frac{2}{2u(1-u)}<0
\]
and 
\[
\theta_{1/2}(d(1-1/2))=0.
\]
For the left endpoint, observe that 
\[
\theta_{u}(R_{u,d})
=(d-1)\log\left(\frac{u}{1-u}\right)
+\frac{R_{u,d}}{2}\log\varphi_u(d)
+\frac{d}{2}\log\psi_u(d),
\]
where
\begin{align*}
\varphi_u(d)&=\frac{(d-du-R_{u,d}u)(du-R_{u,d}u)}{(R_{u,d}u)^2}
=\frac{2d((d-2)u+1)}{u(1-u)(d-2)^2},\\
\psi_u(d)&=\frac{d-du-R_{u,d}u}{du-R_{u,d}u}
=\frac{d(1-u)}{2u((d-2)u+1)}.
\end{align*}
Hence, after some simplifications, we find that 
\[
\frac{{\rm d}^2 \theta_{u}(R_{u,d})}{{\rm d}d^2}=
\frac{\delta_u(d)\log\varphi_u(d)
+\gamma_u(d)}{d(d-2)((d-2)u+1)((d-2)u+d)^3},
\]
where
\begin{align*}
\delta_u(d)=&-4du(1-u)(d-2)((d-2)u+1),\\
\gamma_u(d)=&((d-2)u+d)\big((4u^2-u-1)d^2-2(6u^2-4u+1)d-4u(1-2u)\big).
\end{align*}
Note that the coefficients of the second term in $\gamma_u(d)$ 
are nonpositive for all $0<u\le1/2$.
Hence, observe that $\delta_u(d)<0$ and $\gamma_u(d)<0$ 
for all $d\ge3$ and $0<u\le1/2$.
Note that since 
\[
\frac{{\rm d}\varphi_u(d)}{{\rm d}d}
=-\frac{2((1+2u)d+2(1-2u))}{u(1-u)(d-2)^3}<0
\]
and
\[
\lim_{d\to\infty}\varphi_u(d)=\frac{2}{1-u}\ge 2,
\]
we have that $\varphi_u(d)\ge2$ for all $d\ge3$ and $0<u\le1/2$.
We thus conclude that 
${\rm d}\theta_u(R_{u,d})/{\rm d} d$ is decreasing 
in $d$ for any fixed $0<u\le1/2$. 
Observe that
\[
\lim_{d\to\infty}\frac{{\rm d} \theta_{u}(R_{u,d})}{{\rm d}d}
=-\frac{\log(1-u)+u\log2}{1+u}>0
\]
for all $0<u\le1/2$. 
Therefore $\theta_u(R_{u,d})$ is increasing in $d$, and thus, since
\[
\theta_{u}(R_{u,3})\ge\theta_{1/2}(R_{1/2,3})\approx 0.768
\]
for all $0<u\le1/2$,
we conclude that $\theta_{u}(R_{u,d})\ge0$ for all relevant $d$ and $u$. 
Altogether, $g_d(r,w)$ is increasing in $w$ over ${\cal B}_2$ for any fixed $r$.

Finally, for any $0\leq r\le d(1-u)$, it is easily seen that
\[
\lim_{w\to0^{+}}g_{d}(r,w)=0.
\]

Altogether, $i_{E,u}(d)\ge \widehat A_d(u)$.
\end{proof}

Approximate values for $\widehat  A_d(u)$ are listed in Table~\ref{etable}.

\begin{table}[!h]\footnotesize
\centering
\caption{Approximate values for $\widehat  A_d(u)$. 
By Theorem~\ref{t:eThm}, these
are approximate lower bounds for the $u$-edge 
isoperimetric number $i_{E,u}(d)$.}
\vspace{-0.2cm}
\begin{tabular}{c|ccccccc}
\toprule
$u$ & $\approx \widehat A_3(u)$ & $\approx\widehat A_4(u)$ & $\approx\widehat A_5(u)$ 
& $\approx\widehat A_{10}(u)$ & $\approx\widehat A_{25}(u)$ & $\approx\widehat A_{50}(u)$  
& $\approx\widehat A_{100}(u)$  \\
\hline
$0.01$ & 0.57080 & 1.29152 & 2.07102 & 6.31585 & 20.00259 & 43.58306 & 91.53259  \\
$0.05$ & 0.46150 & 1.06879 & 1.73912 & 5.49362 & 17.96765 & 39.83142 & 84.74426  \\
$0.10$ & 0.39850 & 0.93300 & 1.52904 & 4.91775 & 16.36950 & 36.65008 & 78.56994  \\
$0.15$ & 0.35544 & 0.83739 & 1.37806 & 4.48034 & 15.07700 & 33.96870 & 73.17830  \\
$0.20$ & 0.32140 & 0.76038 & 1.25487 & 4.11019 & 13.93791 & 31.54266 & 68.18908  \\
$0.25$ & 0.29262 & 0.69435 & 1.14821 & 3.78107 & 12.89392 & 29.27612 & 63.45361 \\
$0.30$ & 0.26728 & 0.63557 & 1.05254 & 3.47967 & 11.91538 & 27.12097 & 58.89526  \\
$0.35$ & 0.24435 & 0.58192 & 0.96469 & 3.19830 & 10.98467 & 25.04690 & 54.46794  \\
$0.40$ & 0.22318 & 0.53205 & 0.88263 & 2.93177 & 10.08979 & 23.03451 & 50.13946  \\
$0.45$ & 0.20332 & 0.48501 & 0.80492 & 2.67658 & 9.22247 & 21.06947 & 45.88731  \\
$0.50$ & 0.18447 & 0.44011 & 0.73051 & 2.43002 & 8.37615 & 19.14025 & 41.69360   \\
\bottomrule
\end{tabular}
\label{etable}
\end{table}

With Theorem~\ref{t:eThm} in hand, we can derive lower bounds on 
the asymptotics of $i_{E,u}(d)$ as $d\to\infty$.

\begin{proof}[Proof of Corollary~\ref{c:eCor}]
Put 
\[
\psi(u)=2(1-u)\sqrt{\log(u^{-u}(1-u)^{u-1})}.
\]
For any $c>0$ we have
\[
\widehat f_d(u,u(d(1-u)-c\sqrt{d}))-\log(u^{-u}(1-u)^{u-1})=
\frac{cu\sqrt{d}}{2}\log\mu_d(u,c)-\frac{d}{2}\log\nu_d(u,c),
\]
where
\begin{align*}
\mu_d(u,c) 
&=\frac{(du(1-u)-cu\sqrt{d})^2}{(du^2+cu\sqrt{d})(d(1-u)^2+cu\sqrt{d})},\\
\nu_d(u,c) 
&= 
\left(1-\frac{c}{(1-u)\sqrt {d}}\right)^{2u(1-u)}
\left(1+\frac{c}{u \sqrt{d}}\right)^{u^2}
\left(1+\frac{cu}{(1-u)^2\sqrt{d}}\right)^{(1-u)^2}.
\end{align*}
Observe that
\[
\lim_{d\to\infty}\frac{cu\sqrt{d}}{2}\log\mu_d(u,c)
=-\frac{c^2}{2(1-u)^2}
\]
and
\[
\lim_{d\to\infty}\frac{d}{2}\log\nu_d(u,c)=-\frac{c^2}{4(1-u)^2}.
\]
Hence
\[\lim_{d\to\infty}\widehat f_d(u,u(d(1-u)-c\sqrt{d}))
=-\frac{1}{4(1-u)^2}(c-\psi(u))(c+\psi(u)),
\]
and thus, for any $\epsilon>0$ and sufficiently large $d$,
\[
\widehat A_d(u)\ge d(1-u)-(\psi(u)+\epsilon)\sqrt{d}.
\]
Applying Theorem~\ref{t:eThm}, the result is obtained.
\end{proof}

Note that in the case $u=1/2$, the lower bound of 
Corollary~\ref{c:eCor} agrees with the lower bound at \eqn{eAsymt}.

\section{Acknowledgments}
BK was supported by NSERC of Canada.
NW carried out this research at the 
Department of Combinatorics \& Optimization, 
University of Waterloo, and was supported by the 
Canada Research Chairs program
and NSERC of Canada.
This project began during a visit by BK to the University of Waterloo.
BK would like to thank NW for support and hospitality during this time.



\begin{thebibliography}{11}

\bibitem{Alo97}
\textsc{Alon, N.} (1997).
On the edge-expansion of graphs.
\emph{Combin. Probab. Comput.} \textbf{6} 145\textendash{}152.
\href{http://www.ams.org/mathscinet-getitem?mr=1447810}{MR1447810}

\bibitem{AM85}
\textsc{Alon, N.} and \textsc{Milman, V. D.} (1985). 
$\lambda_{1}$, isoperimetric inequalities for graphs, and superconcentrators. 
\emph{J. Combin. Theory Ser. B} \textbf{38} 73\textendash{}88.
\href{http://www.ams.org/mathscinet-getitem?mr=782626}{MR782626}

\bibitem{AS08}
\textsc{Alon, N.} and \textsc{Spencer, J. H.} (2008). 
\emph{The Probabilistic Method,} 3rd ed. 
Wiley, Hoboken, NJ.
\href{http://www.ams.org/mathscinet-getitem?mr=2437651}{MR2437651}

\bibitem{Ben73}
\textsc{Bender, E. A.} (1973). 
Central and local limit theorems applied to asymptotic enumeration. 
\emph{J. Combinatorial Theory Ser. A} \textbf{15} 91\textendash{}111.
\href{http://www.ams.org/mathscinet-getitem?mr=0375433}{MR0375433}

\bibitem{Bol88}
\textsc{Bollob\'as, B.} (1988). 
The isoperimetric number of random regular graphs. 
\emph{European J. Combin.} \textbf{9} 241\textendash{}244.
\href{http://www.ams.org/mathscinet-getitem?mr=947025}{MR947025}

\bibitem{Bus84}
\textsc{Buser, P.} (1984). 
On the bipartition of graphs. 
\emph{Discrete Appl. Math.} \textbf{9} 105\textendash{}109.
\href{http://www.ams.org/mathscinet-getitem?mr=754431}{MR754431}

\bibitem{DeA02}
\textsc{De Angelis, V.} (2003). 
Asymptotic expansions and positivity of coefficients for large powers 
of analytic functions. 
\emph{Int. J. Math. Math. Sci.} 1003\textendash{}1025. 
\href{http://www.ams.org/mathscinet-getitem?mr=1976089}{MR1976089}

\bibitem{DDSW}
\textsc{D\'{\i}az, J.}, \textsc{Do, N.}, \textsc{Serna, M. J.} 
and \textsc{Wormald, N. C.} (2003).
Bounds on the max and min bisection of random cubic and 
random 4-regular graphs. 
\emph{Theoret. Comput. Sci.} \textbf{307} 531\textendash{}547.
\href{http://www.ams.org/mathscinet-getitem?mr=2021746}{MR2021746}

\bibitem{DSW}
\textsc{D\'{\i}az, J.}, \textsc{Serna, M. J.} and \textsc{Wormald, N. C.} (2007).
Bounds on the bisection width for random $d$-regular graphs. 
\emph{Theoret. Comput. Sci.} \textbf{382} 120\textendash{}130.
\href{http://www.ams.org/mathscinet-getitem?mr=2352107}{MR2352107}

\bibitem{Fri03}
\textsc{Friedman, J.} (2008).
A proof of Alon's second eigenvalue conjecture and related problems.
\emph{Mem.\ Amer.\ Math.\ Soc.} \textbf{195} viii+100.
\href{http://www.ams.org/mathscinet-getitem?mr=2437174}{MR2437174}

\bibitem{Gol94}
\textsc{Golovach, P. A.} (1995). 
On the computation of the isoperimetric number of a graph. 
\emph{Cybernet. Systems Anal.} \textbf{30} 453\textendash{}457.
\href{http://www.ams.org/mathscinet-getitem?mr=1323782}{MR1323782}

\bibitem{HLW}
\textsc{Hoory, S.}, \textsc{Linial, N.} and \textsc{Wigderson, A.} (2006). 
Expander graphs and their applications. 
\emph{Bull. Amer. Math. Soc. (N.S.)} \textbf{43} 439\textendash{}561. 
\href{http://www.ams.org/mathscinet-getitem?mr=2247919}{MR2247919}

\bibitem{KM93}
\textsc{Kostochka, A. V.} and \textsc{Melnikov, L. S.} (1993).
On a lower bound for the isoperimetric number of cubic graphs. 
In \emph{Probabilistic methods in discrete mathematics (Petrozavodsk, 1992). 
Progr. Pure Appl. Discrete Math.} \textbf{1} 251\textendash{}265. VSP, Utrecht.
\href{http://www.ams.org/mathscinet-getitem?mr=1383140}{MR1383140}

\bibitem{L14}
\textsc{Lampis, M.} (2012).
Local improvement gives better expanders.
Preprint available at \href{http://arxiv.org/abs/1211.0524}{arXiv:1211.0524}.

\bibitem{MP01}
\textsc{Monien, B.} and \textsc{Preis, R.} (2006). 
Upper bounds on the bisection width of 3- and 4-regular graphs. 
\emph{J. Discrete Algorithms} \textbf{4} 475\textendash498.
\href{http://www.ams.org/mathscinet-getitem?mr=2258338}{MR2258338}

\bibitem{OR85}
\textsc{Odlyzko, A. M.} and \textsc{Richmond, L. B.} (1985). 
On the unimodality of high convolutions of discrete distributions. 
\emph{Ann. Probab.} \textbf{13} 299\textendash{}306. 
\href{http://www.ams.org/mathscinet-getitem?mr=770644}{MR770644}

\bibitem{SS11}
\textsc{Sauerwald, T.} and \textsc{Stauffer, A.} (2011). 
Rumor spreading and vertex expansion on regular graphs. 
In \emph{Proceedings of the Twenty-Second Annual ACM-SIAM 
Symposium on Discrete Algorithms} 462\textendash{}475. 
SIAM, Philadelphia, PA.
\href{http://www.ams.org/mathscinet-getitem?mr=2857140}{MR2857140}

\bibitem{Tan84}
\textsc{Tanner, R. M.} (1984). 
Explicit concentrators from generalized $N$-gons. 
\emph{SIAM J. Algebraic Discrete Methods} \textbf{5} 287\textendash{}293.
\href{http://www.ams.org/mathscinet-getitem?mr=752035}{MR752035}

\bibitem{Wor99}
\textsc{Wormald, N. C.} (1999). 
Models of random regular graphs.
In \emph{Surveys in Combinatorics, 1999 (Canterbury). London Math. 
Soc. Lecture Note Ser.}
\textbf{267} 239\textendash{}298. Cambridge Univ. Press, Cambridge.
\href{http://www.ams.org/mathscinet-getitem?mr=1725006}{MR1725006}

\end{thebibliography}
\end{document}